\documentclass[11pt,a4paper,reqno]{amsart}
\usepackage{amsfonts}
\usepackage{amsmath,amssymb,amsthm,amsxtra}
\usepackage{float}
\usepackage[colorlinks, linkcolor=RoyalBlue,anchorcolor=Periwinkle,
    citecolor=Orange,urlcolor=Emerald]{hyperref}
\usepackage[usenames,dvipsnames]{xcolor}
\usepackage{enumitem}
\usepackage{geometry} 
\usepackage{graphicx}
\usepackage{subfigure}
\usepackage{tikz}\usetikzlibrary{matrix}
\usepackage{url}

\usepackage[all]{xy}

\theoremstyle{plain}
\newtheorem{theorem}{Theorem}[section]
\newtheorem{lemma}[theorem]{Lemma}
\newtheorem{corollary}[theorem]{Corollary}
\newtheorem{proposition}[theorem]{Proposition}

\theoremstyle{definition}
\newtheorem{definition}[theorem]{Definition}
\newtheorem{example}[theorem]{Example}
\newtheorem{remark}[theorem]{Remark}

\numberwithin{equation}{section}

\def\hua{\mathcal}
\def\kong{\mathbb}
\def\<{\langle}
\def\>{\rangle}

\def\Ind{\operatorname{Ind}}
\def\Sim{\operatorname{Sim}}
\def\Hom{\operatorname{Hom}}
\def\End{\operatorname{End}}

\def\diff{\operatorname{d}}

\def\Ps{\operatorname{Ps}}

\def\deg{\operatorname{deg}}

\newcommand{\h}{\operatorname{\hua{H}}}            

\renewcommand{\k}{\mathbf{k}}
\renewcommand{\mod}{\operatorname{mod}}

\newcommand{\tilt}[3]{{#1}^{#2}_{#3}}

\def\numbers{\begin{enumerate}[label=\arabic*{$^\circ$}.]}
\def\ends{\end{enumerate}}

\newcommand{\Proj}{\operatorname{Proj}}
\newcommand{\EG}{\operatorname{EG}}       
\newcommand{\EGp}{\operatorname{EG}^\circ}       
\newcommand{\C}[2]
{\operatorname{\hua{C}_{#1}(#2)}}               

\newcommand{\CEG}[2]{\operatorname{CEG}_{#1}(#2)}             
\newcommand{\D}{\operatorname{\hua{D}}}

\newcommand{\qq}[1]{\operatorname{\Gamma}_{#1}Q}
\newcommand{\cc}[1]{\Q{#1}}

\def\zero{\hua{H}_\Gamma}
\def\nzero{\hua{H}_Q}

\newcommand{\aq}[1]{\operatorname{\Gamma}_{#1}A_n}
\newcommand{\wan}[1]{\mathfrak{#1}}

\newcommand{\Q}[1]{\mathcal{Q}(#1)}

\renewcommand{\b}{\mathbf{b}}
\def\BT{\operatorname{BT}}
\def\TP{\operatorname{TP}}
\def\G{\operatorname{G}}
\newcommand{\iv}{\operatorname{iv}}

\title{Ext-quivers of hearts of $A$-type and the orientation of associahedron}
\author{{\small Yu Qiu }}
\date{\today}

\begin{document}

\begin{abstract}
We classify the Ext-quivers of hearts
in the bounded derived category $\D(A_n)$
and the finite-dimensional derived category $\D(\aq{N})$ of
the Calabi-Yau-$N$ Ginzburg algebra $\aq{N}$.
This provides the classification
for Buan-Thomas' colored quiver for higher clusters of A-type.
We also give explicit combinatorial constructions from a binary tree with $n+2$ leaves
to a torsion pair in $\mod \k \overrightarrow{A_n}$ and
a cluster tilting set in the corresponding cluster category,
for the straight oriented $A$-type quiver $\overrightarrow{A_n}$.
As an application,
we show that the orientation of the $n$-dimensional
associahedron induced by poset structure of binary trees
coincides with the orientation induced by poset structure of torsion pairs
in $\mod \k \overrightarrow{A_n}$ (under the correspondence above).

\vskip .3cm
{\parindent =0pt
\it Key words:} Ext-quiver, binary tree, torsion pair, cluster theory
\end{abstract}

\maketitle

\section*{Summary}

Assem and Happel \cite{AH1} gave a classification of iterated tilted algebras of $A$-type
using tilting theory decants ago.
In the first part of the paper (Section~1 and Section~2),
we generalize their result to classify (Theorem~\ref{thm:GGT})
the Ext-quivers of hearts of $A$-type (i.e. in $\D^b(A_n)$),
in terms of graded gentle trees.
As an application,
we describe (Corollary~\ref{cor:GCGC}) the Ext-quivers of hearts in $\D(\aq{N})$,
the finite-dimensional derived category
of the Calabi-Yau-$N$ Ginzburg algebra $\aq{N}$,
which correspond (cf. \cite[Theorem~8.6]{KQ}) to colored quivers for $(N-1)$-clusters of A-type,
in the sense of Buan-Thomas \cite{BT}.

In the second part of the paper (Section~3),
we give explicit combinatorial constructions
(Proposition~\ref{pp:PB=TP} and Proposition~\ref{pp:PB=CEG}),
from a binary trees with $n+2$ leaves (for parenthesizing a word with $n+2$ letters)
to a torsion pair in $\mod \k \overrightarrow{A_n}$ and
a cluster tilting sets in the (normal) cluster category $\C{}{A_n}$,
where $\overrightarrow{A_n}$ is a straight oriented $A_n$ quiver.
Thus, we obtain the bijections between these sets.
As an application,
we show (Theorem~\ref{thm:orientation})
that under the bijection above, the orientation of the $n$-dimensional
associahedron induced by poset structure of binary trees (cf. \cite{L})
coincides with the orientation induced by poset structure of torsion pairs
(or hearts, in the sense of King-Qiu \cite{KQ}).

Note that there are many potential orientations
for the $n$-dimensional associahedron, arising from representation theory of quivers,
cf. \cite[Figure~4 and Theorem~9.6]{KQ}).
These orientations are also interested in physics (see \cite{CCV}),
as they are related to wall crossing formula,
quantum dilogarithm identities and Bridgeland's stability condition
(cf. \cite{K6} and \cite{Q}).

\subsection*{Acknowledgements}
I would like to thank Alastair King for introducing me this topic.

\section{Preliminaries}

\subsection{Derived category and cluster category}
Let $Q$ be a quiver of A-type with $n$ vertices and $\k$ a fixed algebraic closed field.
Let $\k Q$ be the path algebra, $\nzero=\mod\k Q$
its module category and let $\hua{D}(Q)=\hua{D}^b(\nzero)$
be the bounded derived category.
Note that $\D(Q)$ is independent of the orientation of $Q$
and we will write $A_n$ for $Q$ sometimes.

Denote by $\tau$ the \emph{AR-functor} (cf. \cite[Chapter IV]{ASS1}).
Let $\C{}{A_n}$ be the \emph{cluster category} of $\D(A_n)$,
that is the orbit category of $\D(A_n)$ quotient by $[1]\circ\tau$.
Denote by $\pi_n$ be the quotient map
\[\pi_n:\D(A_n)\to\C{}{A_n}.\]

\subsection{Calabi-Yau category}
Denoted by $\qq{N}$ the (degree $N$)
\emph{Ginzburg's differential graded algebra} that associated to $Q$
which the \emph{dg algebra}
\[
    \k \big< e, x, x^*, e^* \mid   e\in Q_0; x\in Q_1 \big>
\]
with degrees
\[
    \deg e =\deg x =0,\quad \deg x^* =N-2,\quad \deg e^* =N-1
\]
and differentials
\[
    \diff \sum_{e\in Q_0} e^*=\sum_{x\in Q_1} \, [x_k,x^*_k] .
\]
Let $\D(\qq{N})$ be the finite dimensional derived category of $\qq{N}$
and $\zero$ be its canonical heart.
Notice that the derived categories are always triangulated.
Again, since $\D(\qq{N})$ is independent of the orientation of $Q$,
we will write $\aq{N}$ for $\qq{N}$.

\subsection{Hearts of triangulated categories}
A \emph{torsion pair} in an abelian category $\hua{C}$ is a pair of
full subcategories $\<\hua{F},\hua{T}\>$ of $\hua{C}$,
such that $\Hom(\hua{T},\hua{F})=0$ and furthermore
every object $E \in \hua{C}$ fits into a short exact sequence
$ \xymatrix@C=0.5cm{0 \ar[r] & E^{\hua{T}} \ar[r] & E \ar[r] & E^{\hua{F}} \ar[r] & 0}$
for some objects $E^{\hua{T}} \in \hua{T}$ and $E^{\hua{F}} \in \hua{F}$.

A \emph{t-structure} on a triangulated category $\hua{D}$ is
a full subcategory $\hua{P} \subset \hua{D}$,
satisfying $\hua{P}[1] \subset \hua{P}$ and being the torsion part
of some torsion pair (with respect to triangles) $\<\hua{P},\hua{P}^{\perp}\>$ in $\D$.
A t-structure $\hua{P}$ is \emph{bounded} if
\[
  \hua{D}= \displaystyle\bigcup_{i,j \in \kong{Z}} \hua{P}^\perp[i] \cap \hua{P}[j].
\]
The \emph{heart} of a t-structure $\hua{P}$ is the full subcategory
\[
  \h=  \hua{P}^\perp[1]\cap\hua{P}
\]
and any bounded t-structure is determined by its heart.
In this paper, we only consider bounded t-structures and their hearts.

Recall that we can forward/backward tilts a heart $\h$ to get a new one,
with respect to any torsion pair in $\h$
in the sense of Happel-Reiten-Smal\o (\cite{HRS}, see also \cite[Proposition~3.2]{KQ}).
Further, all forward/backward tilts with respect to torsion pairs in $\h$,
correspond one-one to all hearts between $\h$ and $\h[\pm1]$
(in the sense of King-Qiu \cite{KQ}).
In particular there is a special kind of tilting which is called
simple tilting (cf.\cite[Definition~3.6]{Q}).
We denote by $\tilt{\h}{\sharp}{S}$ and $\tilt{\h}{\flat}{S}$, respectively,
the simple forward/backward tilts of a heart $\h$, with respect to a simple $S$.

The \emph{exchange graph} of a triangulated category $\D$
to be the oriented graph, whose vertices are all hearts in $\D$
and whose edges correspond to the simple forward titling between them.
Denote by $\EG(A_n)$ the exchange graph of $\D(A_n)$,
and $\EGp(\aq{N})$ the principal component of the exchange graph of $\D(\aq{N})$, that is,
the connected component containing $\zero$.

\section{Ext-quivers of A-type}
\subsection{Graded gentle tree}

In \cite{AH1},
it gives the complete description of all iterated tilted algebra of type $A_n$,
namely:

\begin{definition}\cite{ASS1}
Let $A$ be an quiver algebra with acyclic quiver $T_A$.
The algebra $A \cong \k T_A/\hua{I}$ is called \emph{gentle}
if the bound quiver $(T_A, \hua{I})$ has the following properties:
\numbers
\item
    Each point of $T_A$ is the source and the target of at most two arrows.
\item
    For each arrow $\alpha \in (T_A)_1$,
    there is at most one arrow $\beta$ and one arrow $\gamma$
    such that $\alpha\beta \notin \hua{I}$ and $\gamma\alpha \notin \hua{I}$.
\item
    For each arrow $\alpha \in (T_A)_1$,
    there is at most one arrow $\xi$ and one arrow $\zeta$
    such that $\alpha\xi \in \hua{I}$ and $\zeta\alpha \in \hua{I}$.
\item
    The ideal $\hua{I}$ is generated by the paths in $3^\circ$.
\ends
If $T_A$ is a tree,
the gentle algebra $A \cong \k T_A/\hua{I}$ is called
a \emph{gentle tree algebra}
\end{definition}

\begin{theorem}\label{thm:GT2}
Let $A$ be a quiver algebra with bound quiver $(T_A, \hua{I})$.
Then $A$ is (iterated) tilted algebras of type $A_n$
if and only if
$(T_A, \hua{I})$ is a gentle trees algebra.
{\rm(cf.\cite{AH1}, also \cite{ASS1})}
\end{theorem}

Considering the special properties of $T_A$,
we can color it into two colors,
such that any two neighbor arrows $\alpha, \beta$ has the same color
if and only if
$\alpha\beta \in \hua{I}$ or $\beta\alpha \in \hua{I}$.
Alternatively,
we can also color it into two colors,
such that any two neighbor arrows $\alpha, \beta$ has the different color
if and only if
$\alpha\beta \in \hua{I}$ or $\beta\alpha \in \hua{I}$.
By the properties above,
either coloring is unique up to swapping colors.
Hence we have another way to characterize gentle tree algebra as follows.

\begin{definition}
A \emph{gentle tree} is a quiver $T$ with a 2-coloring,
such that each vertex has at most one arrow of each color incoming or outgoing.
\end{definition}

For a colored quiver $T$,
there are two natural ideals
\[
    \hua{I}_T^+:\mbox{generated by all unicolor-paths of length two};
\]
\[
    \hua{I}_T^-:\mbox{generated by all alternating color paths of length two}.
\]

\begin{proposition}
Let $A=\k T/\hua{I}$ be a bound quiver algebra.
We have the following equivalent statement:
\begin{itemize}
\item
    $A$ is a gentle tree algebra.
\item
    $T$ is some gentle tree with $\hua{I}=\hua{I}_T^+$ or $\hua{I}=\hua{I}_T^-$.
\end{itemize}
\end{proposition}
\begin{proof}
By the one of two ways of coloring,
the relations in the ideal and the coloring of the gentle tree
can be determined uniquely by each other.
\end{proof}

\begin{remark}
\label{remark}
In fact, there is an irrelevant but interesting result that
for a gentle tree $T$,
$\k T/\hua{I}_T^+$ and $\k T/\hua{I}_T^-$ are Koszul dual.
\end{remark}

We are going to generalize Theorem~\ref{thm:GT2}
to describe all hearts in $\D(A_n)$.

\subsection{Ext-quivers of hearts}

Recall that a heart $\h$ is a finite, if the set of its simples,
denoted by $\Sim\h$,
is finite and generates $\h$ by means of extensions,

\begin{definition}
Let $\h$ be a finite heart in a triangulated category $\D$
and $\mathbf{S}=\bigoplus_{S\in\Sim\h} S$.
The Ext-quiver $\Q{\h}$ is the (positively) graded quiver
whose vertices are the simples of $\h$ and
whose graded edges correspond to a basis of
$\End^\bullet(\mathbf{S}, \mathbf{S})$.
\end{definition}

Note that, by \cite{KQ},
$\h$ is finite, rigid and strongly monochromatic for any $\h$ in $\D(A_n)$.
By \cite[Lemma~3.3]{KQ}, we know that there are at most one arrow
between any two vertices in $\Q{\h}$.

\begin{definition}
A \emph{graded gentle tree} $\hua{G}$
is a gentle tree with a positive grading for each arrow.
The associated quiver $\cc{\hua{G}}$ of $\hua{G}$,
is a graded quiver with the same vertex set and
an arrow $a:i \to j$ for each unicolored path $p:i \to j$ in $\hua{G}$,
with the grading of $p$.
\end{definition}

Define a mutation $\mu$ on graded gentle tree as follow.

\begin{definition}
For a graded gentle tree $\hua{G}$,
let $V$ be a vertex with neighborhood
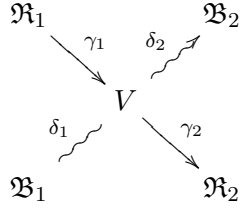
\[
    \xymatrix@C=1.5pc@R=1.5pc{
    \wan{R}_1 \ar[dr]^{\gamma_1} &&  \wan{B}_2 \ar@{<~}[dl]_{\delta_2} \\
    & V & \\
    \wan{B}_1 \ar@{~}[ur]^{\delta_1} &&  \wan{R}_2 \ar@{<-}[ul]_{\gamma_2} \\
    }
\]
where $\wan{B}_i, \wan{R}_i$ are the sub trees
and $\gamma_i, \delta_i$ are degrees, $i=1,2$.
The straight line represent one color
and the curly line represent the other color.
Define the \emph{forward mutation $\mu_V$} at vertex $V$ (on $\hua{G}$) as follows:
\begin{itemize}
\item
    if $\delta_1 \geq 1$,
    $\mu_V$ on the lower part of of the quiver is:
        
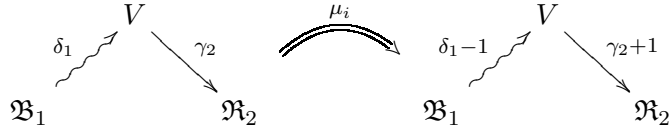
\[\xymatrix@C=1.8pc@R=1.8pc{
        & V & \\
        \wan{B}_1 \ar@{~>}[ur]^{\delta_1} &&  \wan{R}_2 \ar@{<-}[ul]_{\gamma_2} \\
        }
        \xymatrix@C=0.8pc@R=0.8pc{&&&\\ \ar@{=>}@/^1pc/[rrr]^{\mu_i} &&&}
        \xymatrix@C=1.8pc@R=1.8pc{
        & V & \\
        \wan{B}_1 \ar@{~>}[ur]^{\delta_1-1} &&  \wan{R}_2 \ar@{<-}[ul]_{\gamma_2+1} \\
        }
    \]
\item
    if $\delta_1 = 1$, denote
    
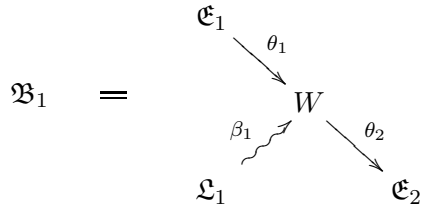
\[
        \xymatrix@C=0.8pc@R=1.5pc{ \\ \wan{B}_1 & \ar@{=}[r] && }
        \xymatrix@C=1.5pc@R=1.5pc{
        \wan{E}_1 \ar[dr]^{\theta_1} &&  \\
        & W & \\
        \wan{L}_1 \ar@{~>}[ur]^{\beta_1} &&  \wan{E}_2 \ar@{<-}[ul]_{\theta_2} \\
        }
    \]
    and $\mu_V$ on the lower part of of the quiver is:
    \begin{gather}
    \label{eq:ggt}
        \xymatrix@C=1.5pc@R=1.5pc{
        \wan{E}_1 \ar[dr]^{\theta_1} &&  V \ar@{<~}[dl]_1 \ar[dr]^{\gamma_2} &\\
        & W && \wan{R}_2 \\
        \wan{L}_1 \ar@{~>}[ur]^{\beta} &&  \wan{E}_2 \ar@{<-}[ul]_{\theta_2} \\}
            \xymatrix@C=0.8pc@R=0.8pc{&&\\ \ar@{=>}@/^.7pc/[rr]^{\mu_i} &&}
        \xymatrix@C=1.5pc@R=1.5pc{
        & V \ar@{<~}[dl]_{\beta} \ar[dr]^1 &&  \wan{E}_2^\times \ar@{<~}[dl]_{\theta_2}\\
        \wan{L}_1 && W & \\
        &  \wan{E}_1^\times \ar@{~>}[ur]^{\theta_1} &&  \wan{R}_2 \ar@{<-}[ul]_{\gamma_2} \\
        }
    \end{gather}
    where $\wan{X}^\times$ is the operation of swapping colors
    on a graded gentle trees $\wan{X}$.
\item
    $\mu_V$ on the upper part follows the mirror rule of the lower part. \\
\end{itemize}
Dually, define the backward mutation $\mu_V^{-1}$ to be the reverse of $\mu_V$
(which follows a similar rule).
\end{definition}

Clearly, the set of all graded gentle trees with $n$ vertexes
is closed under such mutation.
In fact, this set is also connected under (forward/backward) mutation.

\begin{lemma}\label{lem:GGTconn}
Any graded gentle tree with $n$ vertices can be iteratedly mutated from
another graded gentle tree with $n$ vertices.
\end{lemma}
\begin{proof}
Use induction, starting from the trivial case when $n=1$.
Suppose that the lemma follows for $n=m$
and consider the case for $n=m+1$.
We only need to show that any graded gentle tree $\hua{G}$ with $m+1$ vertices
can be iteratedly mutated from an unicolor graded gentle tree with all degrees equal zero.
Let $V$ be a sink in $\hua{G}$ and the subtree of $\hua{G}$ by deleting $V$
is $\hua{G}'$ while the connecting arrow from $\hua{G}'$ to $V$ has degree $d$.
By backward mutating on $V$, we can increase $d$ as large as possible
without changing $\hua{G}'$.
Then the mutation at a vertex other than $V$ on $\hua{G}$
restricted to $\hua{G}'$ will be the same as mutating at that vertex on $\hua{G}'$.
Thus, by the induction assumption, we can mutate $\hua{G}$ such that
$\hua{G}'$ becomes unicolor with all degrees equal zero.
Then, repeatedly forward mutating many times on $V$ will turn $\hua{G}$
into unicolor with all degrees equal zero.
\end{proof}

Using Lemma~\ref{lem:ADC}, a direct calculation gives the following proposition.

\begin{proposition}
\label{pp:commutative}
Let $\hua{G}$ be a graded gentle tree and $\h$ be a heart in $\D(A_n)$.
If $\cc{\hua{G}}=\Q{\h}$
with vertex $V$ in $\hua{G}$ corresponding to the simple $S$ in $\h$,
then
\begin{equation}
\label{eq:mutated}
\Q{\tilt{\h}{\sharp}{S}}=\Q{\mu_V \hua{G}},\quad
\Q{\tilt{\h}{\flat}{S}}=\Q{\mu^{-1}_V \hua{G}}.
\end{equation}
\end{proposition}

Now we can describe all Ext-quiver of hearts of A-type.

\begin{theorem}
\label{thm:GGT}
The Ext-quivers of hearts in $\hua{D}(A_n)$
are precisely the associated quivers of graded gentle trees with $n$ vertices.
\end{theorem}

\begin{proof}
Note that any heart in $\D(A_n)$ can be iterated tilted from
the standard heart $\nzero$, by \cite{KV}.
Without lose of generality, let $Q$ has straight orientation.
Then $\Q{\nzero}$ certainly is the associated quiver for
the graded gentle tree $\hua{G}_Q$ with the same orientation and alternating colored arrow.
Then, inducting from $\nzero$ and using \eqref{eq:mutated},
we deduce that the Ext-quiver of any heart in $\hua{D}(A_n)$
is the associated quivers of some graded gentle tree with $n$ vertices.
On the other hand,
the set of graded gentle trees with $n$ vertices is connected (Lemma~\ref{lem:GGTconn}).
Then, also by induction, we deduce that
the associated quiver of any graded gentle tree with $n$ vertices
is the Ext-quiver of some heart, because \eqref{eq:mutated} and the fact that
we can forward/backward tilt any simples in any heart in $\D(A_n)$ (\cite[Theorem~5.7]{KQ}).
\end{proof}

Recall that we can CY-N double a graded quiver
in the sense of \cite[Definition~6.2]{KQ}.
Then we have the following corollary.

\begin{corollary}\label{cor:GCGC}
The Ext-quivers of hearts in $\EGp(\aq{N})$ are precisely
the CY-N double of the associated quivers of graded gentle trees
with $n$ vertices.
\end{corollary}
\begin{proof}
By \cite[Corollary~8.3]{KQ}, any heart $\h$ in $\EGp(\aq{N})$
is induced from some heart $\h'$ in $\D(A_n)$,
while Ext-quiver $\Q{\h}$ is the CY-N double of $\Q{\h'}$
by \cite[Proposition~7.5]{KQ}.
Thus the corollary follows from Theorem~\ref{thm:GGT}.
\end{proof}

By \cite[Proposition~8.6]{KQ},
the augmented graded quivers of colored quivers for $(N-1)$-clusters
(cf. \cite[Definition~6.1]{KQ} and \cite{BT})
of type $A_n$ are also precisely
the CY-N double of the associated quivers of graded gentle trees.

\section{Associahedron}
\subsection{Binary trees}
Let $\BT_m$ be the set of binary trees with $m+1$ leaves
(and hence with $m$ internal vertices),
which can be used to parenthesize a word with $m+1$ letters
(see Figure~\ref{fig:PB} and cf. \cite{L}).
Let $\G_m$ be the full subgraph of the grid $\kong{Z}^2$ inducing by
\[
    \G_m=\{(x,y)\mid x\geq0, y\geq0, x+y\leq m\}\subset\kong{Z}^2.
\]
It is well-known that a binary tree with $m+1$ leaves has a normal form
as a subgraph of $\G_m$, such that the leaves are $\{(x,m-x)\}_{x=0}^m$,
and we will identify the binary tree with such normal form
(see Figure~\ref{fig:PB}).

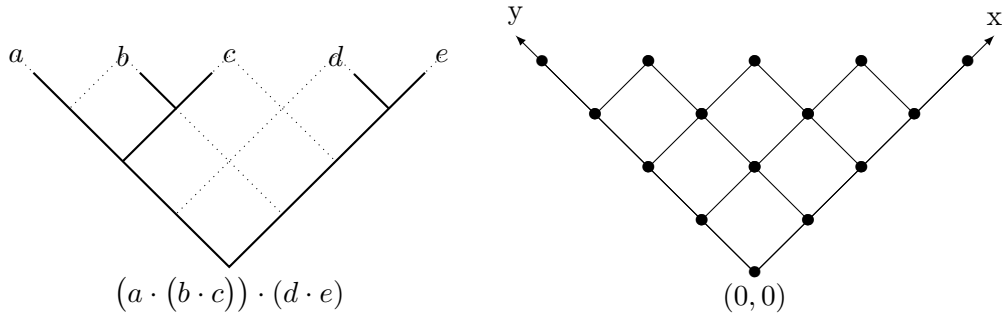
\begin{figure}[h]\centering
\begin{tikzpicture}[scale=.7]
\draw[thick] (-3,3) node (a) {$a$};
\draw[thick] (-1,3) node (b) {$b$};
\draw[thick] (1,3) node (c) {$c$};
\draw[thick] (3,3) node (d) {$d$};
\draw[thick] (5,3) node (e) {$e$};
\draw[thick] (a)  -- (1,-1) -- (e) ;
\draw[thick] (b)  -- (0,2);
\draw[thick] (c)  -- (-1,1);
\draw[thick] (d)  -- (4,2);
\foreach \j in {1,...,4}{
    \foreach \k in {1,...,\j}{
        \draw[dotted] (2*\k-\j-1,\j-1)--(2*\k-\j,\j-2)--(2*\k+1-\j,\j-1);
}}
\draw (1,-1) node[below]
    {$\left( a\cdot \big( b \cdot c \big) \right) \cdot \left( d \cdot e\right)$};
\end{tikzpicture}
\quad
\begin{tikzpicture}[scale=.7]
\foreach \j in {1,...,4}{
    \foreach \k in {1,...,\j}{
        \draw (2*\k-\j,\j)--(2*\k+1-\j,\j-1)--(2*\k+2-\j,\j);
        \draw (2*\k-\j+1,\j-1) node {$\bullet$};
        \draw (2*\k-\j,\j) node {$\bullet$};
        \draw (2*\k+2-\j,\j) node {$\bullet$};
}}
\draw (2,0) node[below]    {$(0,0)$};
\path[->,>=latex] (2,0) edge (-2.5,4.5); \draw (-2.5,4.5) node[above] {y};
\path[->,>=latex] (2,0) edge ( 6.5,4.5); \draw ( 6.5,4.5) node[above] {x};;
\end{tikzpicture}
\caption{A parenthesizing of four words on the left and $\G_4$ on the right.}\label{fig:PB}
\end{figure}
\begin{example}\label{ex:An}
Let
\[
    \G_m^+=\G_m\cap\{(x,y)\mid xy>0\},\quad \G_m^*=\G_m-\{(0,0)\}.
\]
Consider the $A_n$-quiver $\overrightarrow{A_n}: n\to \dots \to 1$ and
let $\h_n=\mod \k \overrightarrow{A_n}$ with corresponding simples $S_1,\ldots,S_n$.
Then, there are canonical bijections (cf. Figure~\ref{fig:G})
\begin{gather*}
    \xi_n: \G_{n+1}^+\to\Ind(\h_n), \\
    \varsigma_n:\G_n^*\to\Ind\h_n\cup \Proj\h_n[1]
\end{gather*}
satisfying $\xi_n(i,j)=\varsigma_n(i-1,j)=M_{i,j}$,
where $M_{i,j}\in\Ind\h_n$ is determined by
\begin{gather}\label{eq:Kgroup}
    [M_{i,j}]=\sum_{i}^{n+1-j} [S_k].
\end{gather}
Let $\zeta_n=\pi_n \circ \varsigma_n:\G_n^*\to\Ind\C{}{A_n}$.
\end{example}

\begin{figure}[t]\centering
\begin{tikzpicture}[scale=.7]
\draw[thick,red] (0,1.5) -- (4, 5.5) -- (-4, 5.5) --cycle;
\draw[thick,blue] (-1,0.5) -- (1, 0.5) -- (5, 4.5) -- (-5,4.5) --cycle;
\foreach \j in {1,...,5}{
    \foreach \k in {1,...,\j}{
        \draw[thin,dashed] (2*\k-\j-2,\j)--(2*\k-\j-1,\j-1)--(2*\k-\j,\j);
        \draw (2*\k-\j-1,\j-1) node {$\bullet$};
        \draw (2*\k-\j-2,\j) node {$\bullet$};
        \draw (2*\k-\j,\j) node {$\bullet$};
}}
\path[->,>=latex] (0,0) edge (-5.5,5.5); \draw (-5.5,5.5) node[above] {y};
\path[->,>=latex] (0,0) edge ( 5.5,5.5); \draw ( 5.5,5.5) node[above] {x};;
\draw (0,5.5) node[above, red] {$\G_5^+ \longleftrightarrow \Ind(\h_4)$};
\draw (4.5,0) node[above, blue] {$\G_4^*\longleftrightarrow\Ind\C{}{A_4}$};
\draw (-3,0) node[above] {$\G_5$};
\end{tikzpicture}
\caption{$\G_5^+$ (red) and $\G_4^*$ (blue) sit inside $\G_5$.}\label{fig:G}
\end{figure}
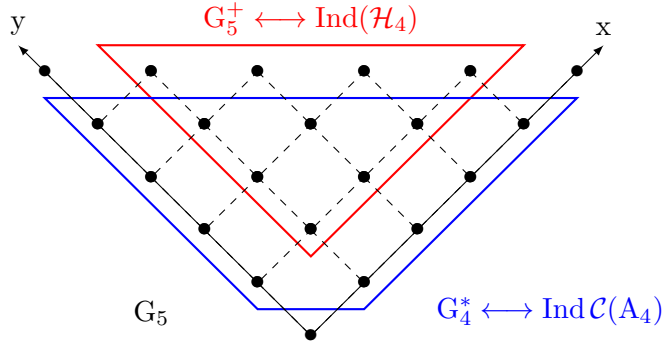

It is known that the following sets
(see \cite{IT} for more possible sets)
can parameterize the vertex set of an $n$-dimensional associahedron:
\numbers
\item
    the set $\BT_{n+1}$ of binary trees with with $n+2$ leaves;
\item
    the set of triangulations of regular $(n+3)$-gon.
\item
    the set $\CEG{}{A_n}$ of (2-)cluster tilting sets in $\C{}{A_n}$;
\item
    the set $\TP(\overrightarrow{A_n})$ of
    torsion pairs in $\h_n$ (cf. \cite{KQ} and \cite{HRS}),
\item
    the set $\EG(\h_n, \h_n[1])$ of hearts in $\D(A_n)$ between $\h_n$ and $\h_n[1]$
    (in the sense of King-Qiu, \cite{KQ}).
\ends
Therefore, there are bijections between these sets.

Furthermore, by \cite[Section~9]{KQ},
the poset structure of torsion pairs (hearts) gives
an orientation $O_t$ of the $n$-dimensional associahedron,
i.e. the orientation of $\EG(\h_n, \h_n[1])$
(considered as a subgraph of $\EG(A_n)$).

On the other hand, there is a poset structure of binary trees,
inducing by locally flipping a binary tree (as shown in Figure~\ref{fig:tree}),
or equivalently, changing the corresponding parenthesizing of words
from $(A\cdot B)\cdot C$ to $A\cdot (B\cdot C)$ (see \cite{L} for details).
This poset structure also gives an orientation $O_p$ for the associahedron.
We aim to prove $O_t=O_p$ this section.

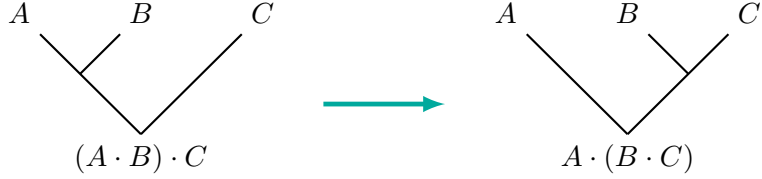
\begin{figure}\centering
\begin{tikzpicture}[scale=.8]
\draw[thick] (-3,3) node (a) {$A$};
\draw[thick] (-1,3) node (b) {$B$};
\draw[thick] (1,3) node (c) {$C$};
\draw[thick] (a)  -- (-1,1);
\draw[thick] (b)  -- (-2,2);
\draw[thick] (c)  -- (-1,1);
\draw (-1,1) node[below] {$(A\cdot B)\cdot C$};
\draw[->,>=latex,line width=.15pc,Emerald] (2,1.5)--(4,1.5);
\draw[thick] (5,3) node (a) {$A$};
\draw[thick] (7,3) node (b) {$B$};
\draw[thick] (9,3) node (c) {$C$};
\draw[thick] (a)  -- (7,1);
\draw[thick] (b)  -- (8,2);
\draw[thick] (c)  -- (7,1);
\draw (7,1) node[below] {$A\cdot (B\cdot C)$};
\end{tikzpicture}
\caption{A local filp of a binary tree (at the word B)}\label{fig:tree}
\end{figure}

\subsection{Combinatorial constructions}
First, we give explicit construction of torsion pairs from binary trees.
For any $p\in\kong{Z}^2$ with coordinate $(x_p, y_p)$,
let $L(p)$ be the edge connecting $(x_p-1, y_p)$ and $p$
and $R(p)$ be the edge connecting $(x_p, y_p-1)$ and $p$.
Define
\begin{gather}
    \hua{T}(\b)=
        \< \xi_n(p) \mid p\in\G_{n+1}^+, L(p)\in\b  \>,\quad
    \hua{F}(\b)=
        \< \xi_n(p) \mid p\in\G_{n+1}^+,R(p)\in\b  \>,
\end{gather}
where $\<\;\>$ means generating by extension.

\begin{proposition}
\label{pp:PB=TP}
There is a bijection $\Theta_n:\BT_{n+1}\to\TP(\overrightarrow{A_n})$,
sending $\b\in\BT_{n+1}$ to $\<\hua{T}(\b), \hua{F}(\b)\>$.
\end{proposition}
\begin{proof}
We only need to show that
$\Theta_n\colon\b\mapsto\<\hua{T}(\b), \hua{F}(\b)\>$
is well-defined (and obviously injective) and hence bijective
since both sets have $n$ elements.

To do so, we first show that any object $M\in\h_n$ admits a short exact sequence
\begin{gather}\label{eq:ses}
    0\to T\to M\to F\to0
\end{gather}
for some $T\in\hua{T}(\b)$ and $F\in\hua{F}(\b)$.
Let $m=\xi_n^{-1}(M)\in\G_{n+1}^+$.
If $m\in\b$ then $M\in\hua{T}(\b)\cup\hua{F}(\b)$ and we have a trivial
short exact sequence\eqref{eq:ses}.
If $m\notin\b$,
let $t$ be the vertex in $\b\cap\{(x_m,j)\mid j\geq y_m\}$ with minimal y-coordinate
and $f$ be the vertex in $\b\cap\{(i,y_m)\mid i\geq x_m\}$ with minimal x-coordinate;
let $a$ and $b$ be the vertices with coordinates
$(n+1-y_t,y_t)$ and $(x_f, n+1-x_f)$,
see Figure~\ref{fig:PB2}.
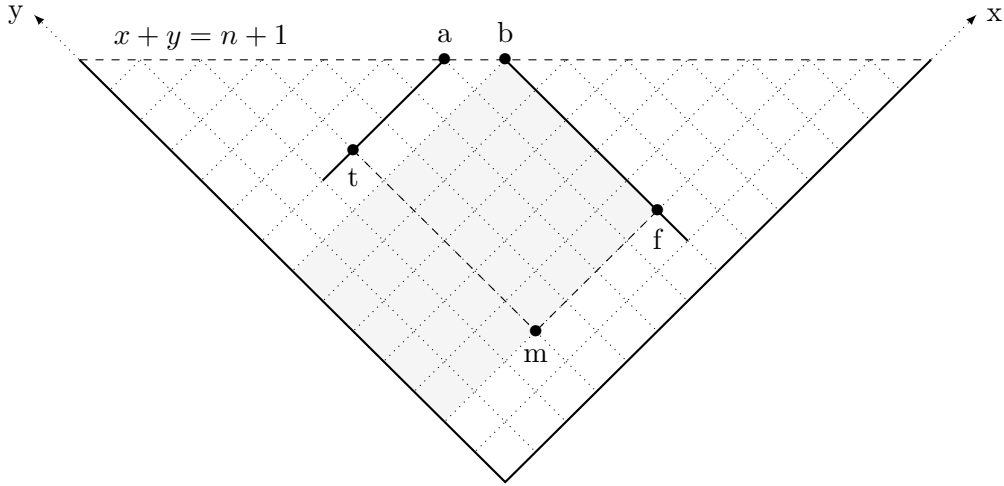
\begin{figure}[b]\centering
\begin{tikzpicture}[scale=.4]
\draw[white,fill=gray!8] (0,14)--(5,9)--(-2,2)--(-7,7)--cycle;
\foreach \j in {1,...,14}{
    \foreach \k in {1,...,\j}{
        \draw[dotted] (2*\k-\j-2,\j)--(2*\k-\j-1,\j-1)--(2*\k-\j,\j);
}}
\path[->,>=latex] (0,0) edge[dotted] (-15.5,15.5);
\draw (-15.5,15.5) node[left] {y};\path[->,>=latex] (0,0) edge[dotted] ( 15.5,15.5);
\draw ( 15.5,15.5) node[right] {x};\draw[thick] (14,14)--(0,0)--(-14,14);
\draw[dashed] (-14,14) -- (14,14);\draw (-10,14) node[above] {$x+y=n+1$};
\draw (1,4.75) node[below] {m};\draw (1,5) node {$\bullet$};
\draw(-5,10.75) node[below] {t};\draw(-5,11) node {$\bullet$};
\draw(-2,14.25) node[above] {a};\draw(-2,14) node {$\bullet$};
\draw(5,8.75) node[below] {f};\draw(5,9) node {$\bullet$};
\draw(0,14.25) node[above] {b};\draw(0,14) node {$\bullet$};
\draw[dashed] (-5,11)--(1,5);\draw[thick] (-5,11)--(-6,10);
\draw[dashed] (5,9)--(1,5);\draw[thick] (5,9)--(6,8);
\draw[thick] (-2,14)--(-5,11);\draw[thick] (0,14)--(5,9);
\end{tikzpicture}
\caption{A short exact sequence in $\G_{n+1}$.}\label{fig:PB2}
\end{figure}
By construction and the property of the binary tree, we know that
\begin{itemize}
\item edges in the line segments, from $m$ to $t$ and from $m$ to $f$, are not in $\b$;
\item edges in the line segments, from $a$ to $t$ and from $b$ to $f$, are in $\b$;
\item
$(x_a,y_a)+(1,-1)=(x_b,y_b)$, i.e. $a,b$ are neighbors in the line $x+y=n+1$;
\item $L(t)$ and $R(f)$ are in $\b$.
\end{itemize}
Thus $T=\xi_n(t)\in\hua{T}(\b)$ and $F=\xi_n(f)\in\hua{F}(\b)$.
By \eqref{eq:Kgroup}, a direct calculation shows that $[M]=[T]+[F]$,
which implies we have \eqref{eq:ses}, by Lemma~\ref{lem:TRI}, as required.

To finish, we need to show that $\Hom(\hua{T}(\b), \hua{F}(\b))=0$.
Let $F=\xi_n(f)\in\hua{F}(\b)$.
As above, edges in the line segments from $b$ to $f$ are in $\b$.
By the property of binary tree,
the horizonal edges (i.e. parallelling to x-axis)
in the shadow area in Figure~\ref{fig:PB2} are not in $\b$,
which implies, by Lemma~\ref{lem:HOM}, that
the modules in $\h_n$ that has nonzero maps to $F$ are not in $\hua{T}(\b)$,
as required.
\end{proof}

Next, we identify cluster tilting sets from binary trees via $\zeta_n$.
For any $\b\in\BT_{n+1}$,
let $\iv(\b)$ be set of the internal vertices expect $(0,0)$
so that $\#\iv(\b)=n$.
Denote by $\Proj\h$ a complete set of indecomposable projectives of a heart $\h$.
Recall (\cite[Section~2]{KQ}) that
\begin{gather}\label{eq:defproj}
     P\in\Proj\h  \iff P\in\Ind(\hua{P}\cap\tau^{-1}\hua{P}^\perp),
\end{gather}
where $\hua{P}$ is the t-structure corresponding to $\h$.

\begin{proposition}
\label{pp:PB=CEG}
Let $\b\in\BT_{n+1}$ and $\h(\b)$ be the heart corresponding to
the torsion pair $\Theta_n(\b)$ in $\h_n$.
Then we have $\Proj\h(\b)=\varsigma_n(\iv(\b))$ and there is a bijection
$\zeta_n\circ\iv: \BT_{n+1}\to\CEG{}{A_n}$.
\end{proposition}
\begin{proof}
By \cite[Corollary~5.12]{KQ}, we know that $\pi_n\Proj\h(\b)\in\CEG{}{A_n}$
and hence the second claim follows immediately from the first one.

Let $p\in\iv(\b)$, which is the intersection of the edges $L(r)$ and $R(q)$,
where $q,r$ be the points with coordinates $(x_p,y_p+1)$ and $(x_p+1,y_p)$
(see Figure~\ref{fig:PB3}).
Note that $p$ is not in the line $x_p+y_p\leq n$ and thus $q,r\in\G_{n+1}$.
Let $\hua{P}(\b)$ be the t-structure corresponding to $\h(\b)$.
Note that
\begin{gather*}
    \hua{P}(\b)=\hua{T}(\b)\cup\bigcup_{j>0}\h_n[j],\quad
    \hua{P}(\b)^\perp=\hua{F}(\b)\cup\bigcup_{j<0}\h_n[j].
\end{gather*}
\begin{figure}[b]\centering
\begin{tikzpicture}[scale=.3]
\foreach \j in {1,...,14}{
    \foreach \k in {1,...,\j}{
        \draw[dotted] (2*\k-\j-2,\j)--(2*\k-\j-1,\j-1)--(2*\k-\j,\j);
}}
\path[->,>=latex] (0,0) edge[dotted] (-15.5,15.5);
\draw (-15.5,15.5) node[left] {y};\path[->,>=latex] (0,0) edge[dotted] ( 15.5,15.5);
\draw ( 15.5,15.5) node[right] {x};\draw[thick] (14,14)--(0,0)--(-14,14);
\draw[dashed] (-14,14) -- (14,14);
\draw(2,7.5) node[below] {p};\draw(2,8) node {$\bullet$};
\draw[thick] (-4,14)--(2,8)--(8,14);
\draw(.5,9) node[below,left] {q};\draw(1,9) node {$\bullet$};
\draw(3.5,9) node[below,right] {r};\draw(3,9) node {$\bullet$};
\end{tikzpicture}
\caption{An interval vertex of a binary tree in $\G_{n+1}$.}\label{fig:PB3}
\end{figure}
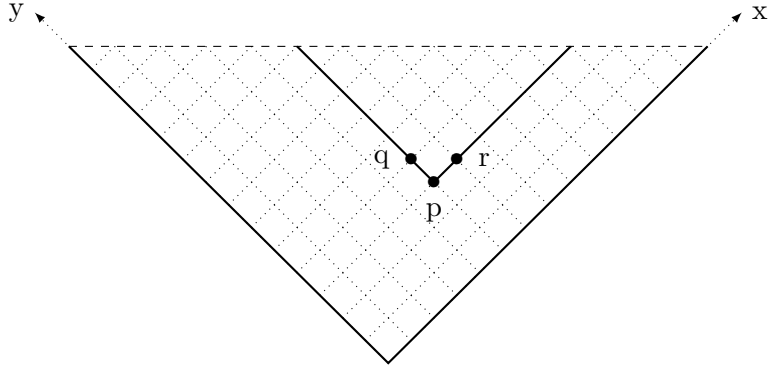

If $r\in\G_{n+1}^+$, then $P=\varsigma_n(p)=\xi_n(r)$ is in $\hua{T}(\b)$;
otherwise, $y_p=0$ and then $P\in\h_n[1]$. Either way, $P\in\hua{P}(\b)$.
Similarly,
if $q\in\G_{n+1}^+$, then $\tau P=\xi_n(q)$ is in $\hua{F}(\b)$;
otherwise, $x_p=0$ and then $\tau P\in\h_n[-1]$. Either way, $\tau P\in\hua{P}(\b)^\perp$.
Therefore $P\in\Proj\h(\b)$ by \eqref{eq:defproj}.
Thus $\Proj\h(\b)$ contains, and hence equals $\varsigma_n(\iv(\b))$ as required,
noticing that $\#\Proj\h(\b)=n=\#\iv(\b)$.
\end{proof}

\begin{example}
Keep the notation in Example~\ref{ex:An}.
Then the binary tree in Figure~\ref{fig:PB} corresponds to the torsion pair
$\hua{T}=\<M_{2,2},M_{1,2}\>, \hua{F}=\<M_{1,3},M_{3,1}\>$
and the cluster tilting set $\{ M_{1,2}, M_{2,2},M_{1,1}[1]  \}$.
\end{example}

\subsection{The orientation}

Now we apply the constructions above to show that $O_t=O_p$.

\begin{theorem}\label{thm:orientation}
Under the bijection $\Theta_n$ in Proposition~\ref{pp:PB=TP},
the orientations $O_t$ and $O_p$ of the $n$-dimensional associahedron coincide.
\end{theorem}
\begin{proof}
Consider an edge $e:\b_1\to\b_2$ in $\BT_{n+1}$,
which corresponds to a local flip as in Figure~\ref{fig:tree}.
Let $\h(\b_i)$ forward tilt of $\h_n$ with respect to $\Theta_n(\b_i)$.
We only need to show that $\h(\b_2)$ is a simple forward tilt of $\h(\b_1)$.

By Proposition~\ref{pp:PB=CEG},
we know that $\Proj\h(\b_i)=\varsigma_n(\iv(\b_i))$, for $i=1,2$, defer by one object.
Denote by $P_i\in\Proj\h(\b_i)$ the different objects.
Thus, $\pi_n\Proj\h(\b_i)\in\CEG{}{A_n}$ are related by one mutation,
which implies $\h(\b_i)$ are related by a single simple tilting, by \cite[Corollary~5.12]{KQ},
and $P_1,P_2$ are related by some triangle
\[
    P_j \to M \to P_k \to P_j[1]
\]
in $\D(A_n)$ for some ordering $\{j,k\}=\{1,2\}$.
By Lemma~\ref{lem:suc+pre}, $P_j$ is a predecessor of $P_k$.
But, from the flip we know that $P_1$ is the predecessor of $P_2$,
which implies $j=1$ and $k=2$.
Thus the forward simple tiling is from $\h(\b_1)$ to $\h(\b_2)$ as required.
\end{proof}

\begin{example}
Figure~\ref{fig:last} is the orientation of the $2$-dimensional assoiahedron,
induced by poset structure of binary trees, which is the oriented pentagon
in \cite[Figure~3]{KQ} and \cite[(3.5)]{Q}, cf. also \cite[Figure~5]{K6}.

\begin{figure}[h]\centering
\begin{tikzpicture}[scale=.4]

\foreach \j in {1,...,3}{
    \foreach \k in {1,...,\j}{
        \draw[thin,dotted] (2*\k-\j-2,\j)--(2*\k-\j-1,\j-1)--(2*\k-\j,\j);
}}
\draw (0,0) node {$\bullet$};
\draw[thick] (-3,3)  -- (0,0) -- (3,3);
\draw[thick] (-1,3)  -- (0,2);
\draw[thick] (1,3)   -- (-1,1);

\foreach \j in {1,...,3}{
    \foreach \k in {1,...,\j}{
        \draw[thin,dotted] (2*\k-\j-2+10,\j)--(2*\k-\j-1+10,\j-1)--(2*\k-\j+10,\j);
}}
\draw (0+10,0) node {$\bullet$};
\draw[thick] (-3+10,3)  -- (0+10,0) -- (3+10,3);
\draw[thick] (-1+10,3)  -- (1+10,1);
\draw[thick] (1+10,3)   -- (0+10,2);

\foreach \j in {1,...,3}{
    \foreach \k in {1,...,\j}{
        \draw[thin,dotted] (2*\k-\j-2-10*.4,\j+10*.88)--(2*\k-\j-1-10*.4,\j-1+10*.88)--(2*\k-\j-10*.4,\j+10*.88);
}}
\draw (0-10*.4,0+10*.88) node {$\bullet$};
\draw[thick] (-3-10*.4,3+10*.88)  -- (0-10*.4,0+10*.88) -- (3-10*.4,3+10*.88);
\draw[thick] (-1-10*.4,3+10*.88)  -- (-2-10*.4,2+10*.88);
\draw[thick] (1-10*.4,3+10*.88)   -- (-1-10*.4,1+10*.88);

\foreach \j in {1,...,3}{
    \foreach \k in {1,...,\j}{
        \draw[thin,dotted] (2*\k-\j-2+10*1.4,\j+10*.88)--(2*\k-\j-1+10*1.4,\j-1+10*.88)--(2*\k-\j+10*1.4,\j+10*.88);
}}
\draw (0+10*1.4,0+10*.88) node {$\bullet$};
\draw[thick] (-3+10*1.4,3+10*.88)  -- (0+10*1.4,0+10*.88) -- (3+10*1.4,3+10*.88);
\draw[thick] (-1+10*1.4,3+10*.88)  -- (1+10*1.4,1+10*.88);
\draw[thick] (1+10*1.4,3+10*.88)   -- (2+10*1.4,2+10*.88);

\foreach \j in {1,...,3}{
    \foreach \k in {1,...,\j}{
        \draw[thin,dotted] (2*\k-\j-2+5,\j+10*1.2)--(2*\k-\j-1+5,\j-1+10*1.2)--(2*\k-\j+5,\j+10*1.2);
}}
\draw (0+5,0+10*1.2) node {$\bullet$};
\draw[thick] (-3+5,3+10*1.2)  -- (0+5,0+10*1.2) -- (3+5,3+10*1.2);
\draw[thick] (-1+5,3+10*1.2)  -- (-2+5,2+10*1.2);
\draw[thick] (1+5,3+10*1.2)   -- (2+5,2+10*1.2);

\draw[->,>=latex,line width=.15pc,Emerald] (2.5,1.2) -- (7.5,1.2);
\draw[->,>=latex,line width=.15pc,Emerald] (1.5-10*.4,10*.88)--(0,4);
\draw[->,>=latex,line width=.15pc,Emerald] (2.5-10*.4,1.2+10*.88) -- (3,1+10*1.2);
\draw[<-,>=latex,line width=.15pc,Emerald] (10-1.5+10*.4,10*.88)--(10,4);
\draw[<-,>=latex,line width=.15pc,Emerald] (10-2.5+10*.4,1.2+10*.88) -- (10-3,1+10*1.2);

\end{tikzpicture}
\caption{The orientation of the $2$-dimensional assoiahedron}
\label{fig:last}
\end{figure}
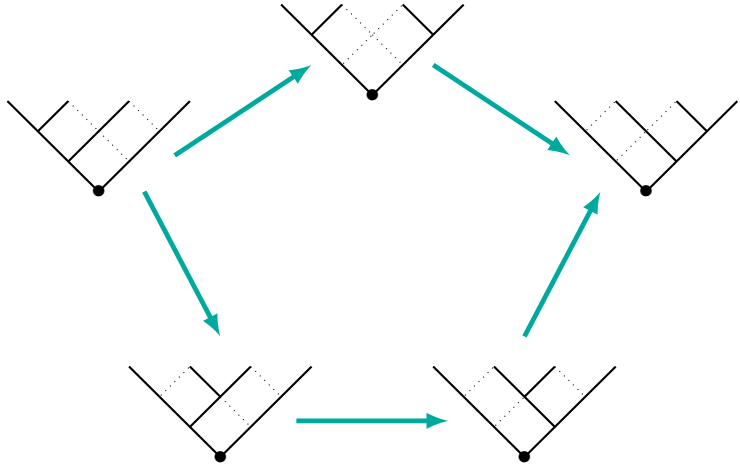
\end{example}

\appendix
\section{Maps and triangles in D(An)}\label{app}
In this appendix, we collet several facts about
the maps and triangles in $\D(A_n)$.
See \cite[Chapter~IX]{ASS1} for the proofs of the first three lemmas.

Recall there are notions of sectional paths and predecessors in $\D(A_n)$
cf. \cite[Section~2.2]{Q}.

\begin{lemma}
\label{lem:TRI}
Let $M,A,B\in \Ind\hua{D}(A_n)$ such that
$A\in\Ps^{-1}(M)$ and $B\in\Ps(M)-\Ps(A)$.
Then there is a short exact sequence $0\to A\to M\to B\to0$ if and only if
$[M]=[A]+[B]$.
\end{lemma}

\begin{lemma}
\label{lem:HOM}
Let $M,L \in \Ind\hua{D}(A_n)$.
Then $\Hom(M,L) \neq 0$ if and only if
\begin{gather*}
    L \in \Big[\Ps(M),\Ps^{-1} \big( \tau(M[1]) \big) \Big],\quad
    M \in \Big[ \Ps \big( \tau^{-1}(L[-1]) \big), \Ps^{-1}(L) \Big].
\end{gather*}
\end{lemma}

\begin{lemma}\label{lem:suc+pre}
If $\Hom(L,M[1])\neq0$ for some $M$ and $L$ in $\Ind\D(A_n)$,
then $M$ is a predecessor of $L$.
Any two non-isomorphic indecomposables in $\D(A_n)$ can not be
predecessors of each other.
\end{lemma}

\begin{lemma}\label{lem:ADC}
Let $\h$ be a heart in $\D(A_n)$.
If there are the following full sub-quivers
\[
    \xymatrix@C=1.2pc{
    &S\ar[dr]^a\\
    T \ar[ur]^1 \ar[rr]_{a+1} && A}
\quad
    \xymatrix@C=1.2pc{
    &S\ar[dr]^b\\
    T \ar[ur]^1 && B}
\quad
    \xymatrix@C=1.2pc{
    &S\ar@{<-}[dr]^{c+1}\\
    T \ar[ur]^1 && C\ar[ll]_{c}}
\quad
    \xymatrix@C=1.2pc{
    &S\ar@{<-}[dr]^{d+1}\\
    T \ar[ur]^1 && D}
\]
in the Ext-quiver $\Q{\h}$
for some $S,T, A,B,C,D\in\Sim\h$ and positive integer $a,b,c,d$,
then there are following full sub-quivers
\[
    \xymatrix@C=1pc{
    &S[1]\ar[dr]^{a+1}\\
    R \ar@{<-}[ur]^1 && A}
\quad
    \xymatrix@C=1pc{
    &S[1]\ar[dr]^{b+1}\\
    R \ar@{<-}[ur]^1 \ar[rr]_{b} && B}
\quad
    \xymatrix@C=1pc{
    &S[1]\ar@{<-}[dr]^{c}\\
    R \ar@{<-}[ur]^1 && C}
\quad
    \xymatrix@C=1pc{
    &S[1]\ar@{<-}[dr]^{d}\\
    R \ar@{<-}[ur]^1 && D\ar[ll]^{d+1}}
\]
in the Ext-quiver $\Q{\tilt{\h}{\sharp}{S}}$,
where $R$ is the nontrivial extension of $T$ on top of $S$.
\end{lemma}
\begin{proof}
We only prove the first case while the other cases are similar.
By \cite[Theorem~5,7]{KQ}, we know that the simples in $\tilt{\h}{\sharp}{S}$
corresponding to $S, T$ and $A$ are $S[1], R$ and $A$.
By \cite[Lemma~3.3]{Q},
we have an isomorphism $\Hom^1(T,S)\otimes\Hom^a(S,A)\to\Hom^{a+1}(T,A)$.
Thus, applying $\Hom(-,A)$ to the triangle $S \to R \to T \to S[1]$
gives $\Hom^\bullet(R,A)=0$.
Similarly, a direct calculation of other $\Hom^\bullet$ between $S[1],R, A$
shows the the new sub-quiver is as required.
\end{proof}



\begin{thebibliography}{99}

\bibitem{AH1}
  I.~Assem \and D.~Happel,
  Generalized tilted algebras of type $A_n$,
  \emph{Communications in Algebra}, \textbf{9(20)}, 2010-2125 (1981).

\bibitem{ASS1}
  I.~Assem, D.~Simson \and A.~Skowroski,
  Elements of the Representation Theory of Associative Algebras 1,
  \emph{Cambridge Uni. Press}, 2006.

\bibitem{BT}
  A.~Buan \and H.~Thomas,
  Coloured quiver mutation for higher cluster categories,
  \emph{Adv. Math.} 222 (2009), 971--995,
  (\href{http://arxiv.org/abs/0809.0691}{arXiv:0809.0691v3}).

\bibitem{CCV}
  S.~Cecotti, C.~Cordova, \and C.~Vafa,
  Braids, Walls, and Mirrors,
  \href{http://arxiv.org/abs/1110.2115}{arXiv:1110.2115v1}.

\bibitem{HRS}
  D.~Happel, I.~Reiten, \and S.Smal\o,
  Tilting in abelian categories and quasitilted algebras,
  \emph{Mem. Amer. Math. Soc.} 120 (1996), no. 575, viii+ 88.

\bibitem{IT}
    C.~Ingalls \and H.~Thomas,
    Noncrossing partitions and representations of quivers,
    \emph{Compos. Math}, 145 (2009), no. 6, 1533¨C1562.

\bibitem{K6}
  B.~Keller,
  On cluster theory and quantum dilogarithm,
  \href{http://arxiv.org/abs/1102.4148}{arXiv:1102.4148v4}.

\bibitem{KV}
  B.~Keller \and D.~Vossieck,
  Aisles in derived categories,
  \emph{Bull. Soc. Math. Belg.} 40 (1988), 239-253.

\bibitem{KQ}
  A.~King \and Y.~Qiu,
  Exchange graphs of acyclic Calabi-Yau categories,
  \href{http://arxiv.org/abs/1109.2924}{arXiv:1109.2924v2}.

\bibitem{L}
    Jean-Louis Loday,
    \href{http://www.claymath.org/programs/outreach/academy/LectureNotes05/Loday.pdf}
    {Associahedron},
    Clay Mathematics Institute, 2005 Colloquium Series.

\bibitem{Q}
  Y.~Qiu,
  Stability spaces and quantum dilogarithms for (Calabi-Yau) Dynkin quivers,
  \href{http://arxiv.org/abs/1111.1010}{arXiv:111.1010v2}.

\bibitem{ST}
  P.~Seidel \and R.~Thomas,
  Braid group actions on derived categories of coherent sheaves,
  \emph{Duke Math. J.}, 108(1):37-108, 2001.
  (\href{http://arxiv.org/abs/math/0001043}{arXiv:math/0001043v2})

\end{thebibliography}
\end{document}